\documentclass[a4paper,11pt]{article}
\usepackage[T1]{fontenc}
\usepackage[utf8]{inputenc}
\usepackage{authblk}
\usepackage[left=3.6cm, right=3.6cm, top=4.1cm, bottom=3.1cm]{geometry}
\usepackage{amsthm,amsmath,amssymb,bm,showlabels} 
\usepackage{stmaryrd}
\usepackage{yhmath}
\allowdisplaybreaks
\usepackage{mathrsfs}
\usepackage{mathdots,mathtools}

\usepackage{fancyhdr}
\usepackage{graphicx,subfigure,caption,epstopdf}
\graphicspath{{figure/}}
\usepackage{stmaryrd}
\usepackage{color}
\usepackage{enumerate}
\usepackage{enumitem}
\usepackage[framemethod=tikz]{mdframed} 
\usepackage{lipsum}

\usepackage{setspace,indentfirst}
\usepackage{float}
\usepackage{geometry} 
\usepackage[square, comma, sort&compress, numbers]{natbib}

\usepackage{changes}
\usepackage{hyperref}
\usepackage{chngcntr}
\usepackage{titlesec}
\usepackage{bm}
\usepackage[most]{tcolorbox}
\usepackage[hang,flushmargin]{footmisc}
\usepackage{caption}
\allowdisplaybreaks[4]

\definechangesauthor{red}
\colorlet{Changes@Color}{red}
\hypersetup{colorlinks=true,citecolor={black},linkcolor={black}}

\usepackage{letltxmacro} 
\LetLtxMacro{\amsmathcal}{\mathcal}
\usepackage{boondox-cal} 
\LetLtxMacro{\mathcal}{\amsmathcal} 
\DeclareMathAlphabet{\mathbcal}{U}{BOONDOX-cal}{m}{n}
\DeclareMathAlphabet{\mathcal}{OMS}{cmsy}{m}{n} 

\newtheorem{theorem}{Theorem}[section]

\newtheorem{definition*}{Definition}
\newtheorem{theorem*}{Theorem}

\newtheorem{lemma*}{Lemma}
\newtheorem{proposition}[theorem]{Proposition}

\newtheorem{corollary*}{Corollary}
\theoremstyle{definition}

\newtheorem{definition}[theorem]{Definition}

\numberwithin{equation}{section}
\numberwithin{figure}{section}

\makeatletter

\newcommand{\Rmnum}[1]{\expandafter\@slowromancap\romannumeral #1@}
\makeatother

\allowdisplaybreaks[4]

\hypersetup{
	colorlinks=true,
	linkcolor=blue,
	filecolor=blue,
	urlcolor=blue,
	citecolor=blue,
}

\definecolor{uibpur}{HTML}{8A2BE2}
\definecolor{uiborg}{HTML}{FF4500}

\tcbuselibrary{breakable}
\newtcolorbox{mybox}{enhanced, boxrule=0pt, frame hidden, sharp corners, colback=blue!5!white, colframe=blue!5!white,breakable,left=0mm,right=0mm,top=0mm,bottom=0mm}

\DeclareSymbolFont{ugmL}{OMX}{mdugm}{m}{n}
\DeclareMathAccent{\wideparen}{\mathord}{ugmL}{"F3}

\titleformat*{\section}{\large\bfseries}
\titleformat*{\subsection}{\normalsize\bfseries}
\makeatletter
\g@addto@macro\bfseries{\boldmath}
\makeatother

\title{{\textbf{An Extension and Refinement of \\the Brouwer-Schauder-Tychonoff \\Fixed Point Theorem}}}
\author[1]{Lixin Cheng}
\author[1]{Chulei Liu}
\author[1]{Wen Zhang\thanks{Corresponding author: Wen Zhang.}}
\affil[1]{School of Mathematical Sciences, Xiamen University, Xiamen, 361005, China}

\date{}

\begin{document}
\maketitle
\footnote{\hspace{1.8em}E-mail addresses: 		
	lxcheng@xmu.edu.cn (L. Cheng),
	chuleiliu@stu.xmu.edu.cn (C. Liu),\\
	\hspace*{9.4em}
	wenzhang@xmu.edu.cn (W. Zhang).}
\footnote{\hspace{1.8em}Supported by the National Natural Science Foundation of China, No. 12271453 and\\
	\hspace*{7.8em}
	Fujian Provincial Natural Science Foundation of China No. 2024J01026.}
\vspace{-20pt}
\begin{abstract}
In this paper, we present the Brouwer-Schauder-Tychonoff fixed point theorem on locally convex spaces as the following extension and improvement:
	Suppose that $S$ is a compact star-shaped subset with respect to $p\in S$ with its convexity index $\alpha_p>0$. Then every continuous self-mapping $f:S\rightarrow S$ has one of the following two properties:\vspace{-2pt}
	\begin{itemize}
		\setlength{\itemsep}{-12pt}
		\item[(a)] The point $p$ is a fixed point of $f$, i.e.,  $f(p)=p$;\\
		\item[(b)] $f$ has uncountably many different eigenvalues and eigenvectors; that is, there exists an injective mapping $\lambda\mapsto x_\lambda$ from  $(0,1]$ into $S$  such that \[f(x_\lambda)=p+\frac{1}{\lambda \alpha_p}(x_\lambda-p).\]
	\end{itemize}
Note that a closed bounded star-shaped set in a locally convex space is convex if and only if $\alpha=1$, and we extend a Brouwer's type fixed-point theorem on compact star-shaped sets in Banach spaces in a more concise manner to locally convex spaces, thereby this is a simplification and an improvement of the Tychonoff fixed-point theorem to compact star-shaped sets.
\end{abstract}

\section{Introduction}
{It is well known that Brouwer's fixed point theorem \cite{Br1}  states that every continuous self mapping on a nonempty compact convex set in $\mathbb R^n$ has a fixed point, and it has promoted the development of many mathematical fields (See, for instance, \cite{Ca1}). In 1930, Schauder \cite{Sc1} elegantly extended Brouwer's theorem to compact convex subsets of an arbitrary Banach space. After that, Tychonoff \cite{Ty1} further extended it to compact convex subsets of locally convex spaces. In 1955, Klee \cite{K1} presented a partial converse of Tychonoff's theorem, namely that any convex subset of a locally convex space with the fixed point property must be compact. Thus, compactness becomes an essential assumption for the discussion of the fixed point property. It prompted many scholars to consider non-convex compact sets. See, for example, \cite{Ja1}, \cite{Pe1} and \cite{Ra1}.

As an important class of non-convex sets, star-shaped sets naturally attract attention. Many scholars have explored the existence of fixed points of non-expansion mappings on them.(See, for example, \cite{Do1}, \cite{Ha1}, \cite{Ha2} and Section 19.5 in \cite{Ha3}). Does there exist a fixed point of any continuous self-mapping on a compact star-shaped set? The answer is no. Bing \cite{Bi1} provided a continuous self-mapping on Knill's cone \cite{Kn1} that lacks a fixed point. 
It has prompted increased attention on some stronger class of star-shaped sets. 

\cite{Ch1} proposed convexity index $\alpha$ of star-shaped subsets in a Banach space, such that a closed bounded star-shaped subset is convex if and only if $\alpha=1$, and further, it was pointed out that the collection of all compact star-shaped subsets with positive convexity index is dense within the collection of all compact star-shaped subsets in the sense of Hausdorff distance. By using tools such as paving, metric projection, and renorming in Banach spaces, \cite{Ch1} presented the following Brouwer-type fixed-point theorem on compact star-shaped sets in Banach spaces.
\begin{theorem}[\cite{Ch1}]
	Suppose that $X$ is a Banach space. If $S$ is a compact  star-shaped subset  with respect to $p\in S$  with convexity index $\alpha_p>0$, then every continuous self-mapping $f:S\rightarrow S$ has one of the following two properties:\vspace{-2pt}
	\begin{itemize}
		\setlength{\itemsep}{-12pt}
		\item[(a)] The point $p$ is a fixed point of $f$, i.e.,  $f(p)=p$;\\
		\item[(b)] $f$ has uncountably many different eigenvalues and eigenvectors; that is, there exists an injective mapping $\lambda\rightarrow x_\lambda$ from  $(0,1]$ into $S$  such that \[f(x_\lambda)=p+\frac{1}{\lambda \alpha_p}(x_\lambda-p).\]
	\end{itemize}
\end{theorem}
Note that convexity index of star-shaped sets is an algebraic notion independent of norms. It led us to expand the consideration scope of the above to locally convex spaces. In this paper, we achieve this aim concisely and give the desired result.}

\section{Convexity Index of Star-Shaped Sets}
{
In this section, we introduce convexity index of star-shaped subsets in locally convex spaces and briefly discuss their properties. In this paper, all topological spaces are assumed to be Hausdorff.

Recall that a subset $S$ of a linear topological space $X$ is said to be star-shaped with respect to  $p\in S$ if for all $x\in S$, the segment $[p,x]\equiv\{\lambda x+(1-\lambda)p: \lambda\in[0,1]\}\subset S$. We use $\mathfrak C(S)$ to denote the set of all points $p\in S$ such that $S$ is star-shaped with respect to  $p$. If, in addition, $S$ is closed in $X$, then we can define the following extended real-valued function $\rho_{S,p}$ by
\begin{equation*}
\rho_{S,p}(x)=\inf\{r>0:x-p\in r(S-p)\},\qquad x\in X
\end{equation*}	
In this case, we say that $\rho_{S,p}$ is the function generated by $S-p.$ It is easy to check that $\rho_{S,\rho}$ is an extended real-valued nonnegative positively homogeneous function.
\begin{definition}[Convexity index]
	Let $S$ be a nonempty star-shaped set of a linear topological space $X$, and $C$ be the closed convex hull of $S$, i.e., $C=\overline{\text{co}}(S)$. Then
	\begin{equation*}
	\alpha_p(S)=\inf_{x\in C}\frac{\rho_{C,p}(x)}{\rho_{S,p}(x)}\notag
	\end{equation*}	
	is called the\textit{ convexity index of $S$ with respect to $p\in \mathfrak{C}(S)$}, and
	$$\alpha(S)=\sup_{p\in\mathfrak{C}(S)}\alpha_p$$
	is called the \textit{convexity index of $S$}.
\end{definition}
\noindent We shall often denote $\alpha_p(S)$ and $\alpha(S)$ simply by $\alpha_p$ and $\alpha$, respectively, whenever no confusion is possible.

Convexity index of star-shaped sets have the following geometric representation.
\begin{proposition}\label{Prop2.2}
	Suppose that $S$ is a star-shaped subset of a locally convex space and $p\in\mathfrak{C}(S)$. Then
	$\alpha_p=\sup\{r\in[0,1]:r(\overline{\mathrm{co}}(S)-p)\subset S-p\}.$
\end{proposition}
\begin{proof}
	Set $C=\overline{\mathrm{co}}(S)$.
	Note that
	$ \alpha_p\rho_{S,p}(x)\leq\rho_{C,p}(x)$ for each $x\in C$.
	It follows that $\alpha_p(C-p)\subset S-p$. Conversely, $r\in[0,1]$ with $r(C-p)\subset S-p$ implies that
	$$ \frac{\rho_{C,p}(x)}{\rho_{S,p}(x)}\geq r,\qquad \forall x\in C$$
	Thus, $\alpha_p\geq r$.
\end{proof}

Let $X$ be a locally convex space. Denote the collection of all nonempty bounded star-shaped subsets of $X$ by $\mathfrak{S}(X)$. Assume that $\mathcal{U}=\{U_{\iota}\}_{\iota\in I}$ is a local base of convex $0$-neighborhoods. For every $\iota\in I$, denote the Minkowski function with respect to $U_{\iota}$ as $\rho_{\iota}$. Note that $\rho_{\iota}$ is a semi-norm on $X$. Fix $\iota\in I$, and define $d_{\iota}:\mathfrak{S}(X)\times\mathfrak{S}(X)\rightarrow [0,\infty)$ by
$$ d_{\iota}(A,B)=\max\left\{\sup_{a\in A}\inf_{b\in B}\rho_{\iota}(a-b),\sup_{b\in B}\inf_{a\in A}\rho_{\iota}(a-b)\right\},\qquad A,B\in \mathfrak{S}(X).$$
It is easy to see that $d_{\iota}$ is a pseudometric on $\mathfrak{S}(X)$. Denote the topology on $\mathfrak{S}(X)$ induced by $\{d_{\iota}\}_{\iota\in I}$ as $\tau_H$. We show below that, in $\tau_H$-topological sense, bounded star-shaped subsets with positive convexity index are very numerous.
\begin{proposition}\label{Prop2.3}
	The collection of all nonempty bounded star-shaped subsets with positive convexity index is dense in $(\mathfrak{S}(X),\tau_H)$.
\end{proposition}
\begin{proof}
	Let $S\in\mathfrak{S}(X)$ and set $C=\overline{\mathrm{co}}(S)$. Pick $p\in\mathfrak{C}(S)$ and fix a finite subset $I_0 \subset I$. Since $C$ is bounded, for any finite set $\{\varepsilon_{\iota}\}_{\iota\in I_0}\subset(0,\infty)$, there exists $t>0$ such that $2t(C-p)\subset \bigcap_{\iota\in I_0}\varepsilon_{\iota}U_{\iota}$. Put $S'=S\cup(t(C-p)+p)$. Clearly, $S'\in\mathfrak{S}(X)$ and $\overline{\mathrm{co}}(S')-p\subset (1+t)\cdot(C-p)$, and further $\alpha(S')\geq\alpha_p(S')>0$ by Proposition \ref{Prop2.2}.
	Note that
	\begin{align*}
	d_{\iota}(S,S')&=\sup_{s'\in {p}\cup (S'\backslash S)}\inf_{s\in S}\rho_{\iota}(s'-s)\leq \sup_{s'\in {p}\cup (S'\backslash S)}\rho_{\iota}(s'-p)\\
	&\leq\sup_{s'\in t(C-p)+p}\rho_{\iota}(s'-p)\leq\sup_{c\in t(C-p)}\rho_{\iota}(c)<\varepsilon_{\iota},
	\end{align*}
	for $\iota\in I_0$. The proof is complete.
\end{proof}
}
\section{Main Result}
{The following theorem is the main result  of this paper.
\begin{theorem}\label{Theo3.1}
	Suppose that $S$ is a compact star-shaped subset of a locally convex space $X$ with respect to $p\in S$ with convexity index $\alpha_p>0$, then every continuous self-mapping $f:S\rightarrow S$ has one of the following two properties:\vspace{-2pt}
	\begin{itemize}
		\setlength{\itemsep}{-12pt}
		\item[(a)] The point $p$ is a fixed point of $f$, i.e.,  $f(p)=p$;\\
		\item[(b)] $f$ has uncountably many different eigenvalues and eigenvectors; that is, there exists an injective mapping $\lambda\mapsto x_\lambda$ from  $(0,1]$ into $S$  such that \[f(x_\lambda)=p+\frac{1}{\lambda \alpha_p}(x_\lambda-p).\]
	\end{itemize}
\end{theorem}
\begin{proof}
	Let $C=\overline{\mathrm{co}}(S)$. By Proposition \ref{Prop2.2}, $\alpha_p(C-p)\subset S-p\subset C-p$. Given $\lambda\in(0,1]$, consider the continuous self-mapping $g_{\lambda}:\lambda\alpha_p(C-p)\rightarrow \lambda\alpha_p(C-p)$ defined by
	$$ g(x)=\lambda\alpha_p(f(x+p)-p).$$
	By Tychonoff's fixed point theorem, there exists an $z_{\lambda}\in \lambda\alpha_p(C-p)$ with $g(z_{\lambda})=z_{\lambda}$. Let $x_{\lambda}=z_{\lambda}+p$. Note that $\lambda\alpha_p(f(x_{\lambda})-p)=x_{\lambda}-p$, and further $$f(x_{\lambda})=p+\frac{1}{\lambda\alpha_p}(x_{\lambda}-p).$$
	
	If $f(p)\neq p$, we claim that $x_{\lambda}=x_{\lambda'}$ implies $\lambda=\lambda'$.
	Otherwise, assume that $x_{\lambda}=x_{\lambda'}$ where $0<\lambda<\lambda'\leq 1$. Note that
	$$ x_{\lambda}=f(x_{\lambda})=p+\frac{1}{\lambda\alpha_p}(x_{\lambda}-p)=p+\frac{1}{\lambda'\alpha_p}(x_{\lambda'}-p)=f(x_{\lambda'})=x_{\lambda'}.$$
	Then $(\lambda-\lambda')(x_{\lambda}-p)=0$, and further $x_{\lambda}=p$. Thus, $f(p)=p$, which contradicts the hypothesis. From this, we have demonstrated that the mapping $\lambda\mapsto x_{\lambda}$ is injective.
\end{proof}
The following result is a a consequence of Theorem \ref{Theo3.1}.

\begin{theorem}\label{Theo1.1}
	Suppose that $K$ is a nonempty compact convex set of a locally convex space $X$, and that $f:K\rightarrow K$ is a continuous mapping. If $f$ is not the identity, then there is a point $p\in K$ such that  $f$ has uncountably many different eigenvalues and eigenvectors; that is, there exists an injective mapping $\lambda\mapsto x_\lambda$ from  $(0,1]$ into $S$  such that \[f(x_\lambda)=p+\frac{1}{\lambda}(x_\lambda-p).\]
\end{theorem}
}

\bibliographystyle{plain}

\end{document}